\documentclass[a4paper, 11pt]{article}

\usepackage{amsmath, amsthm, amssymb}
\usepackage[textheight=24cm]{geometry}
\usepackage[utf8]{inputenc}
\usepackage{hyperref}
\usepackage{authblk}

\usepackage[table]{xcolor}
\newtheorem{theorem}{Theorem}[section]
\numberwithin{theorem}{section}
\newtheorem{lemma}[theorem]{Lemma}
\newtheorem{corollary}[theorem]{Corollary}
\newtheorem{proposition}[theorem]{Proposition}
\theoremstyle{definition}

\newtheorem{definition}[theorem]{Definition}
\newtheorem{example}[theorem]{Example}
\newtheorem{question}[theorem]{Question}

\newcommand{\Z}{{\mathbb{Z}}}
\newcommand{\BH}{\operatorname{BH}}

\newcommand{\wt}{\operatorname{wt}}
\newcommand{\Gal}{\operatorname{Gal}}

\begin{document}

\title{A generalisation of bent vectors for Butson Hadamard matrices\footnote{A preliminary version of this paper was presented by J. A. Armario at the RICCOTA conference held in Rijeka in July 2023, entitled \textit{Self-dual Butson bent sequences}.}}
\date{}
\author[1]{J. A. Armario}
\author[2]{R. Egan}
\author[3]{H. Kharaghani}
\author[4]{P. \'O Catháin}
\affil[1]{\small
Departamento de Matem\'atica Aplicada I,
Universidad de Sevilla,
Avda. Reina Mercedes s/n,
41012 Sevilla, Spain}
\affil[2]{\small School of Mathematical Sciences,
Dublin City University, Ireland}
\affil[3]{\small Department of Mathematics and Computer Science, University of Lethbridge, Lethbridge, Alberta, T1K 3M4, Canada}
\affil[4]{\small Fiontar \& Scoil na Gaeilge,
Dublin City University, Ireland}

\maketitle
\begin{abstract}
An $n\times n$ complex matrix $M$ with entries in the $k^{\textrm{th}}$ roots of unity which satisfies $MM^{\ast} = nI_{n}$ is called a Butson Hadamard matrix. While a matrix with entries in the $k^{\textrm{th}}$ roots typically does not have an eigenvector with entries in the same set, such vectors and their generalisations turn out to have multiple applications. A bent vector for $M$ satisfies $M{\bf x} = \lambda {\bf y}$ where ${\bf x}$  has entries in the $k^{\textrm{th}}$ roots of unity and all entries of $\textbf{y}$ are complex numbers of norm $1$. Such a bent vector ${\bf x}$ is self-dual if ${\bf y} = \mu{\bf x}$ and conjugate self-dual if ${\bf y} = \mu\overline{\bf x}$ for some $\mu$ of norm $1$.
%\marginpar{\tiny \textcolor{blue}{I suggest to change the red sentence for: "A bent vector for $M$ satisfies $M{\bf x} = \lambda {\bf y}$ where ${\bf x}$  has entries in the $k^{\textrm{th}}$ roots of unity and all entries of $\textbf{y}$ are complex numbers of norm $1$." }} 

Using techniques from algebraic number theory, we prove some order conditions and non-existence results for self-dual and conjugate self-dual bent vectors; using tensor constructions and Bush-type matrices we give explicit examples. We conclude with an application to the covering radius of certain non-linear codes generalising the Reed Muller codes. \\
\textbf{MSC2020: 05B20, 94B25, 94A60}\\
\textbf{Keywords: Hadamard matrix, bent function, Butson code}
\end{abstract}

\thispagestyle{empty}

\section{Introduction}

Let $H$ be a matrix of order $n$ with complex entries of modulus $1$. Hadamard's celebrated theorem on maximal determinants states that $\det(HH^{\ast}) \leq n^{n}$, for a recent survey see \cite{Maxdet}. Equality holds if and only if the rows of $H$ are pairwise orthogonal under the Hermitian inner product. In this case $H$ is called a \textit{Hadamard matrix}. Let $\zeta_{k} = e^{\frac{2 \pi \sqrt{-1}}{k}}$ be a primitive $k^{\textrm{th}}$ root of unity, and denote by $\langle \zeta_{k}\rangle$ the set of all $k^{\textrm{th}}$ roots of unity. An $n \times n$  Hadamard matrix with entries in  $\langle \zeta_{k}\rangle$ is called a \textit{Butson Hadamard matrix} of order $n$ and phase $k$. We write $\BH(n,k)$ for the set of such matrices. Note that the terms \textit{Hadamard matrix} and \textit{complex Hadamard matrix} are often used to refer to the elements of $\textrm{BH}(n, 2)$ and $\textrm{BH}(n, 4)$ respectively. The following result is well-known, being essentially contained in the work of Sylvester, \cite{Sylvester}.

\begin{proposition}
Let $G$ be an abelian group of order $n$ and exponent $k$.
The character table of $G$, denoted by $F(G)$, belongs to $\BH(n,k)$.
\end{proposition}

The proof follows immediately from the orthogonality relations for characters of abelian groups. The direct product of groups corresponds to the Kronecker product of character tables, and Sylvester's name is associated in particular with the character tables of elementary abelian $2$-groups:
\[ F(C_{2}^{m}) = \otimes^{m} \begin{pmatrix} 1 & 1 \\ 1 & -1 \end{pmatrix}\,,\]
often called \textit{Sylvester matrices} in the literature. Two ${\pm 1}$-vectors ${\bf x}$ and ${\bf y}$ of length $2^{m}$ are called \emph{bent} if $F(C_{2}^{m}){\bf x} = 2^{m/2}{\bf y}$. The $\pm1$-vector $\textbf{x}$ is \textit{self-dual bent} if and only if $F(C_{2}^{m})\textbf{x} = \pm2^{m/2} \textbf{x}$; equivalently $\textbf{x}$ is an eigenvector of $F(C_{2}^{m})$.

Perhaps the simplest non-real examples of Butson matrices are the discrete Fourier transform matrices $F(C_{n}) =[\zeta_n^{ij}]_{0 \leq i,j \leq n-1}\in \BH(n,n)$. The character tables of elementary abelian groups of odd order, $F(C_{p}^{m}) = \otimes^{m}F(C_{p})$ are sometimes called \textit{generalised Sylvester matrices}. For an abelian group $G$, our convention is that $F(G)$ has \textbf{rows} indexed by characters and \textbf{columns} indexed by group elements. The ordering is arbitrary, but we typically take the natural ordering on $C_{n}$ inherited from $\mathbb{Z}$, and a lexicographic ordering for direct products of cyclic groups. The matrix $F(G)$ acts naturally on column vectors which have entries labelled by the elements of $G$. There is a natural identification between functions $f: G \rightarrow \mathbb{C}$ and column vectors $[f(g)]_{g \in G}$. Note that $F(G)[f(g)]_{g \in G} = [ \langle \chi, \overline{f} \rangle]_{\chi \in \hat{G}}$. For $G = C_{n}$, this is the discrete Fourier transform, while for $G = C_{2}^{n}$ this is the Walsh-Hadamard transform. 

As noted above, a Boolean function $f : C_{2}^{m} \rightarrow C_{2}$ is naturally identified with the vector $\textbf{x} = [(-1)^{f(a)}]^{\top}_{a \in C_{2}^{m}}$. Historically, bent vectors were first studied in cryptography as a class of Boolean functions which are particularly difficult to approximate with affine functions. In the 1990s, connections between bent functions, group-invariant Hadamard matrices and (relative) difference sets have been studied, \cite{Horadam}. Generalising the index-set of the bent vector leads naturally to the following.

\begin{definition}\label{def:bent}
Let $H \in \BH(n,k)$, and let $\textbf{x}$ be a vector of length $n$ with entries in $\langle \zeta_{k}\rangle$. We say that $\textbf{x}$ is $H$-\textit{bent} if and only if
\[ H \textbf{x} = \sqrt{n} \textbf{y}\,,\]
where all entries of $\textbf{y}$ are complex numbers of norm $1$. Furthermore, \textbf{x} is \textit{self-dual $H$-bent} if $\textbf{y} = \lambda \textbf{x}$ for $\lambda$ of modulus $1$; and \textit{conjugate self-dual $H$-bent} if $\textbf{y} = \lambda \overline{\textbf{x}}$, where $\overline{\textbf{x}}$ is the vector of complex conjugates of elements of $\textbf{x}$.
\end{definition}

A special case of this definition was studied for real Hadamard matrices and for $BH(n, 4)$ matrices by Sol\'e, Cheng and coauthors \cite{SCGR21, SLCKS23} under the name of {\em bent sequence attached to a Hadamard matrix}. A bent vector for $F(G)$ is naturally identified with a function $f:G \rightarrow \mathbb{C}$, and such functions were considered by Schmidt for certain maps between abelian groups, \cite{Sch19}. It has been shown computationally that there exist Butson-Hadamard matrices which admit conjugate self-dual bent vectors but do not admit self-dual bent vectors, \cite{SLAEOS23}. 

In this paper, we consider self-dual and conjugate self-dual bent vectors for matrices $H \in \BH(n,k)$. Section \ref{preliminaries} recalls preliminaries and related work. Section \ref{algebraic} applies techniques from algebraic number theory to prove some necessary conditions and non-existence results for self-dual bent and conjugate-bent vectors. In Section \ref{constructions} we give explicit constructions for bent vectors, including a novel construction in prime squared dimension for solutions to the matrix equation $MN = p\overline{N}$ for $M,N \in \BH(p^{2},p)$. Moreover, if $p=3$, we then have $M=N$, and
%\marginpar{\tiny \textcolor{blue}{$M^{2} = \overline{M}$ should be $M\;N= p\overline{N}$ for $M,N\in \BH(p^2,p)$. If $p=3$ then we have $M=N$. RE: agreed, changed.}}
every column of $M$ is a bent vector for $M$; we call such matrices self-bent. Finally, in Section \ref{codes} we give an application to the covering radius of certain codes which generalise the first order Reed-Muller codes.

\section{Preliminaries and Notation}\label{preliminaries}

 For non-zero $k \in\mathbb{Z}$ define $k_{p}$ to be the highest power of $p$ dividing $k$, and $\phi(k)$ to be the Euler phi function. 
 For a complex number $z$, denote by $\overline{z}$ its complex conjugate and by $\mathfrak{R}(z)$ its real part. 
 Denote by $\Z_k$ be the ring of integers modulo $k$ with $k>1$, and denote by $\Z_k^m$ the set of $m$-tuples over $\Z_k$.
The transpose of a matrix $M$ is denoted $M^{\top}$, the conjugate transpose by $M^{\ast}$ and the entrywise application of the 
complex conjugate by $\overline{M}$. 

Throughout this paper we use boldface lower-case for vectors, \textbf{x}, with the entry indexed by $i$ denoted $x_{i}$. We use capital letters for matrices, $M$, with the entry in row $i$ and column $j$ denoted $m_{i,j}$. Indices are understood from context to be drawn either from the set $\{1, \ldots, n\}$, for an arbitrary Hadamard matrix of order $n$, or from a finite abelian group $G$ for $F(G)$, as above. 

%A (necessarily square) matrix is \textit{monomial} if each row and column contains a single non-zero entry.
%We denote the set of $n \times n$ matrices with entries in a set $S$ by $\mathcal{M}_{n}(S)$ (and in general, the set of $m\times n$ matrices by $\mathcal{M}_{m,n}(S)$). For any matrix $M \in \mathcal{M}_{m,n}(\mathbb{C})$ we denote its conjugate transpose by $M^{\ast}$.  Denote the set of monomial matrices in $\mathcal{M}_{n}(S)$ by $\mathrm{Mon}_{n}(S)$. We use $\mathbb{C}^{\times}$ to denote the complex unit circle. Let $\overline{a}$ denote the complex conjugation of the complex number $a$ and let $\mathcal{R}(a)$ its real part. By abuse of notation, for any matrix or vector $X = [x_{ij}]$, we let $\overline{X} = [\overline{x_{ij}}]$. While we consistently use column vectors in this paper, previous literature has used a variety of conventions, we will these point out where appropriate.

\subsection{Butson Hadamard matrices}\label{sec:Butson}

A \textit{Butson Hadamard} matrix of order $n$ and phase $k$
is a matrix $H$ with entries in $\langle \zeta_{k} \rangle$
such that $HH^*=nI_n$, where $I_n$ denotes the identity 
matrix of order $n$. We write $\BH(n,k)$ for the set of such 
matrices; many examples are furnished by the generalised 
Sylvester matrices introduced above. 
A pair of monomial matrices $(P,Q)$ with non-zero entries in $\langle \zeta_{k} \rangle$ acts on the set $\BH(n, k)$ by $H\cdot (P,Q) = PHQ^{\ast}$. This action preserves the set $\BH(n,k)$, and matrices in the same orbit of the action are said to be Hadamard-equivalent.

A Butson matrix $H\in \BH(n,k)$ is conveniently represented in logarithmic form, that is, the matrix $H=[\zeta_k^{\varphi_{i,j}}]_{i,j=1}^n$ is represented by the matrix $L(H)=[\varphi_{i,j}\mod k]_{i,j=1}^n$ with the convention that $L_{i,j}\in\Z_k$ for all $i,j\in\{1,\ldots,n\}.$
\begin{example} \label{formaLog}
The following is a matrix $H \in \BH(4,8)$, displayed in  logarithmic form
\[
L(H)=\left[\begin{array}{cccc}
   0 & 0 & 0 & 0    \\
   0 & 2 & 4 & 6     \\
   0 & 4 & 0 & 4 \\
    0 & 6 & 4 & 2
\end{array}
\right]
\]
\end{example}
Observe that the matrix above is in dephased form, that is, its first row and column are all $0$. Every matrix $H\in \BH(n, k)$ is Hadamard equivalent to a dephased matrix. Throughout this paper all matrices are assumed to be dephased, with the exception of circulant matrices.

\subsection{Bent functions and generalizations}

The initial study of bent vectors was motivated by applications in cryptography and signal processing. A well-known result of Kumar, Scholtz, and Welch uses bilinear forms to construct bent vectors for the generalised Sylvester matrices $F(C_{k}^{m})$, for all $k$ when $m$ is even; and for all $k\neq 2 \mod 4$ when $m$ is odd. When $m$ is odd and $k \equiv 2 \bmod 4$ no examples are known; partial non-existence results have been obtained, \cite{LeungSchmidt19}.

\begin{proposition}[Theorem 1, \cite{KSW85}]\label{ex_sd-bent}
Let $k$ be an integer and $m$ an even integer. Define $f\colon {\mathbb{Z}}_{k}^m \rightarrow {\mathbb{Z}}_{k}$ by
\[
f(x_1,\ldots,x_{2t})=x_1 x_{t+1}+\ldots+x_tx_{2t}\,.
\]
Let ${\bf x}$ be the $C_{k}^{m}$-indexed vector with $x_{c} = \zeta_k^{f(c)}$. 
Then $F(C_{k}^{m}){\bf x}=k^{m/2}\overline{{\bf x}}$. In other words, ${\bf x}$ is  conjugate self-dual $F(C_{k}^{m})$-bent.
\end{proposition}

\subsection{Algebraic number theory}

Let $\mathbb{Q}(\zeta_k)$ denote the {\em cyclotomic field} generated by $\zeta_k$ over the rationals. It is well known that if $z\in \mathbb{Q}(\zeta_k)$ is a root of the unity then $z \in \langle \zeta_{k} \rangle$ if $k$ is even, and $z \in \langle \zeta_{2k}\rangle$ if $k$ is odd. Thus for odd $k$, the fields $\mathbb{Q}(\zeta_k)$ and $\mathbb{Q}(\zeta_{2k})$ coincide, but cyclotomic fields are otherwise non-isomorphic.

The elements of the ring
\[\mathbb{Z}[\zeta_k]=\Big\{\sum_{j=0}^{k-1}b_j\zeta_k^j \,\colon \,b_j\in \mathbb{Z}\Big\}\]
are called {\em cyclotomic integers}, and they are the algebraic integers in the cyclotomic field $\mathbb{Q}(\zeta_k)$, \cite[Theorem 10.2]{Neukirch}. The following is an immediate corollary.

\begin{proposition} \label{prop:cycint}
An algebraic integer of norm $1$ in $\mathbb{Z}[\zeta_{k}]$ is of the form $\pm \zeta_{k}^{t}$ for some integer $t$.
\end{proposition}

Every proper ideal of $\mathbb{Z}[\zeta_k]$ can be uniquely factorized into a product of finitely many prime ideals. The principal ideal of $\mathbb{Z}[\zeta_k]$ generated by $a\in \mathbb{Z}[\zeta_k]$ is denoted by $a\mathbb{Z}[\zeta_k]$. A prime ideal $\mathfrak{p}$ of $\mathbb{Z}[\zeta_k]$ appears in the factorization of $a\mathbb{Z}[\zeta_k]$ if and only if $a\in \mathfrak{p}$. Our first goal in this section is to describe the (well-known) decomposition of the principal ideals $p\mathbb{Z}[\zeta_{k}]$ into prime ideals where $p$ is a rational prime. Recall that $p$ is said to be \textit{ramified} in $\mathbb{Z}[\zeta_{k}]$ if a proper power of a prime ideal divides $p\mathbb{Z}[\zeta_{k}]$, and unramified otherwise. For any composite number $k = p^{r}m$ with $m$ coprime to $p$, recall that $k_{p} = p^{r}$ the $p$-part of $k$.

\begin{theorem}\cite[Section 13.2]{IR90}\label{cycprime} 
Let $p$ be a rational prime, and let $\phi$ be the Euler totient function. Let $f$ be the least positive integer such that $p^{f} \equiv 1 \mod (k/k_{p})$ and define $g$ by $fg = \phi(k/k_{p})$.
The factorisation of $p$ in $\mathbb{Z}[\zeta_{k}]$ is as follows:
\[ p\mathbb{Z}[\zeta_k]=(\mathfrak{p}_1\cdots\mathfrak{p}_g)^{\phi(k_{p})}\, \]
where the $\mathfrak{p}_{i}$ are distinct prime ideals of $\mathbb{Z}[\zeta_{k}]$. In particular,
an odd prime $p$ is ramified in $\mathbb{Z}[\zeta_{k}]$ precisely when $p \mid k$,
and $2$ is ramified when $4 \mid k$.
\end{theorem}

By $\textrm{Gal}_{k} = \mbox{Gal}(\mathbb{Q} (\zeta_k)/\mathbb{Q})$, we denote the Galois group of the extension $\mathbb{Q}(\zeta_k)/\mathbb{Q}$. Each $\sigma\in \mbox{Gal}_{k}$ raises $\zeta_k$ to a power relatively prime to $k$ and is uniquely determined by this action. For any rational prime $p$ and any $\sigma \in \Gal_{k}$, all orbits of $\sigma$ on the prime ideals over $p$ are of equal length. We denote by  $\sigma^*$ the distinguished field automorphism $\mathbb{Q}(\zeta_k)$ which acts as complex conjugation, i.e., $\sigma^*(\zeta_k)=\zeta_k^{-1}$. There is a long history of using algebraic number theory to deduce non-existence results for difference sets and related objects. The following result is crucial (though not always recorded explicitly).

\begin{theorem}[VI.15, \cite{BJL}]\label{BJLThm}
Let $\mathfrak{p}$ be a prime ideal over the rational prime $p$, and let $m = k/k_{p}$ for integer $k$. A field automorphism $\sigma \in \textrm{Gal}_{k}$ fixes $\mathfrak{p}$ if and only if $\sigma(\zeta_{m}) = \zeta_{m}^{p^{j}}$ for some integer $j$. In particular, $\mathfrak{p}$ is fixed by complex conjugation if and only if $\sigma^{*}(\zeta_{m}) = \zeta_{m}^{p^{j}}$ for some integer $j$.
\end{theorem}

\begin{definition}
The prime $p$ is \textit{self-conjugate} modulo $k$ if and only if there exists an integer $j$ such that $p^{j} \equiv -1 \mod k/k_{p}$.
More generally, for integers $n$ and $k$, we say that $n$ is {\it self-conjugate} modulo $k$ if all prime divisors $p$ of $n$ are self-conjugate modulo $k/k_{p}$.
\end{definition}

Clearly by Theorem \ref{BJLThm}, a prime $p$ is self-conjugate modulo $k$ if and only if every prime ideal above $p$ in $\mathbb{Z}[\zeta_{k}]$ is fixed by complex conjugation. This idea is much used in design theory to show non-existence of solutions to certain norm-equations. For example, consider the ideal $5\mathbb{Z}[\zeta_{13}]$, which is unramified and a product of three prime ideals by Theorem \ref{cycprime}, each fixed by complex conjugation by Theorem \ref{BJLThm} since $5^{2} \equiv -1 \mod 13$. Let $\mathfrak{p}$ be a prime ideal dividing $5\mathbb{Z}[\zeta_{13}]$. In the norm equation $xx^{\ast} = 5^{2t+1}$, the right hand side is divisible by an odd power of $\mathfrak{p} \mid 5$, while the left side must be divisible by an even power. Hence no such $x$ exists. One concludes that there are no matrix solutions of $MM^{\ast} = 5I_{n}$ with entries in $\mathbb{Z}[\zeta_{13}]$. 

Suppose that $p$ is self-conjugate modulo $k$, that $p$ divides $n$ and that there exists $H \in \BH(n,k)$. Let $\mathfrak{p}_{i}$ be the prime ideals above $p$ fixed by complex conjugation, where $1 \leq i \leq g$. 
%\marginpar{\tiny \textcolor{blue}{What  is $e$ in $1 \leq i \leq e$??? Should   $e$ be $g$ of Theorem \ref{BJLThm}??? Yes, fixed. Hangover from old notation I think}}
Denote by $x$ the determinant of $H$. Then $xx^{\ast} = n^{n}$. Since $\mathfrak{p}_{i}$ is a prime ideal which divides $n$, it divides one of $x$ and $x^{\ast}$. But if $\mathfrak{p}$ divides $x$ then $\mathfrak{p}^\ast$ divides $x^{\ast}$ and since $\mathfrak{p}_{i}^{\ast} = \mathfrak{p}_{i}$ we conclude that $\mathfrak{p}_{i}^{2}$ divides $n$. Since the conclusion holds for every $\mathfrak{p}_{i}$, it follows that $p$ divides $n^{n}$ to an even power. 

\section{Nonexistence of bent vectors for Butson matrices}\label{algebraic} 

In this section we apply results from algebraic number theory to produce non-existence results for bent vectors of Butson matrices under number theoretic conditions.

\subsection{Order restrictions for $k= 2,3,4$}

We begin with non-existence results for bent vectors for small values of $k$. Note that the next result applies in particular to self-dual and conjugate self-dual $H$-bent vectors.

\begin{proposition}\label{neccesaryconditionE3}
Suppose that $H \in \BH(n,3)$ and that ${\bf x}$ is $H$-bent. 
If there exists an index $i$ such that $(H{\bf x})_{i} \in \langle \zeta_{3} \rangle$ then $n = 9m^{2}$ for integer $m$.
\end{proposition}

\begin{proof}
By hypothesis $(H{\bf x})_{i}$ has the form
\[ s_0 + s_1\zeta_3 + s_2 \zeta_3^2 = \sqrt{n}\,\zeta_3^j \]
where $s_0+s_1+s_2=n$ and $j\in{0,1,2}$. Multiplying both sides by a suitable power of $\zeta_{3}$ we may assume that $s_0=\min\{s_0,s_1,s_2\}$. Since $1 = -\zeta_{3}-\zeta_{3}^{2}$ the former sum reduces to
\[ (s_1-s_0)\zeta_3 + (s_2-s_0) \zeta_3^2 = \sqrt{n}\,\zeta_3^j\,.\]
This equation is satisfied if either $s_1=s_0$ or $s_2=s_0$. In the first case $s_1=s_0$ and $s_2>s_0$. This leaves two equations:
\[\begin{array}{rl}
    (s_2-s_0)^2= & n,  \\
     2 s_0+s_2= & n.
\end{array}\]
The first identity says that $n$ should be a square. Substituting $s_{2} = n-2s_{0}$ yields
\[
(n-3s_{0})^{2} = n,
\]
for which the solutions are $s_0=\displaystyle\frac{n\pm\sqrt{n}}{3}$. Since $s_{0}$ is necessarily an integer, it follows that $n= 9m^{2}$ for integer $m$. The other case is similar.
\end{proof}

The proof of the next result is analogous, and is hence omitted. Note that every real Hadamard matrix belongs also to $\BH(n, 4)$.

\begin{proposition}[cf. \cite{SCGR21}]
Suppose that $H \in \BH(n,4)$ and that ${\bf x}$ is $H$-bent. 
If there exists an index $i$ such that $(H{\bf x})_{i} \in \langle \zeta_{4} \rangle$ then $n = 4m^{2}$ for integer $m$.
\end{proposition}

\subsection{Further non-existence results} 

In this subsection we explore the self-conjugacy condition to deduce 
strong restrictions on bent vectors. 

\begin{theorem}\label{pro_nonexisten_bentsequences}
Suppose that $k$ is an integer, that ${\bf x}$ is $H$-bent for $H\in \BH(n,k)$. If $k_{p} = 1$, and $n_{p} > 1$ and $p$ is self-conjugate modulo $k$ then $n_{p}$ is a square.
\end{theorem}

\begin{proof}
The ideal $p\mathbb{Z}[\zeta_{k}]$ is unramified, say that it is the product of distinct prime ideals $\mathfrak{p}_{i}$ for $1\leq i \leq g$. Since $p$ is self-conjugate modulo $k$, each $\mathfrak{p}_{i}$ is fixed by complex conjugation.

Consider the prime ideal factorisation of both side of $(H{\bf x})_{j} (H{\bf x})_{j}^{\ast} = n$. Since $\mathfrak{p}_{i}$ divides $(H{\bf x})_{j}$ to the same power as $\mathfrak{p}_{i}^{\ast} = \mathfrak{p}_{i}$ divides $(H{\bf x})_{j}^{\ast}$, in particular the power of $\mathfrak{p}_{i}$ dividing the left hand side is even. This holds for every $\mathfrak{p}_{i}$, and since $p$ is unramified, $p$ must divide the left hand side to an even power. But then $n_{p}$ must be an even power of $p$, and the result follows.
\end{proof}

Taking $p = 2$ yields the following result.

\begin{corollary}
Suppose that $k \equiv 2 \mod 4$ and that ${\bf x}$ is $H$-bent for $H\in \BH(n,k)$. If $2$ is self-conjugate modulo $k$ then $n_{2}$ is a square.
\end{corollary}

Finally, we use self-conjugacy to deduce properties of ${\bf y}$ from ${\bf x}$. 

\begin{theorem}\label{teorema_principal}
Let $H\in\BH(n,k)$ and let ${\bf x}$ be $H$-bent with entries in $\langle \zeta_{k} \rangle$, and let ${\bf y} = \frac{1}{\sqrt{n}}H{\bf x}$. If $n$ is self-conjugate modulo $k$, then every entry of ${\bf y}$ belongs to $\langle \zeta_{2k} \rangle$ if $k$ is even and $\langle \zeta_{4k} \rangle$ if $k$ is odd. 
\end{theorem}

\begin{proof}
From the definition of an $H$-bent vector,
\[ y_{i}y_{i}^{\ast} = \left( \sum_{j=1}^{n} h_{ij} x_{j}\right)\sigma^{\ast} \left( \sum_{j=1}^{n} h_{ij} x_{j}\right) = n \]
By hypothesis, $n$ is self-conjugate modulo $k$, so every prime ideal above $n$ in $\mathbb{Z}[\zeta_{k}]$ is fixed by complex conjugation. Hence, $y_{i}\mathbb{Z}[\zeta_{k}] = y_{i}^{\ast}\mathbb{Z}[\zeta_{k}]$. Consequently, we have the following equality of ideals:
\[ y_{i}^{2} \mathbb{Z}[\zeta_{k}] = n\mathbb{Z}[\zeta_{k}]\,. \]
Two principal ideals are equal if and only if they differ by a unit. The only units in $\mathbb{Z}[\zeta_{k}]$ are roots of unity by Proposition \ref{prop:cycint} and hence $y_{i}^{2}\zeta_{k}^{j} = n$ for some $j$. It follows that $\sqrt{y_{i}^{2}/n} = \pm \zeta_{2k}^{j}$ is also a root of unity, since its square is $\pm \zeta_{k}^{j}$. Note that $-\zeta_{k} \in \langle \zeta_{2j}\rangle$ if $k$ is odd, and $-\zeta_{k} \in \langle \zeta_{k}\rangle$ if $k$ is even. This concludes the proof. 
\end{proof}

\subsection{Non-existence of circulant Butson matrices}

We continue with an application to circulant and group-invariant matrices. Recall that an $n \times n$ matrix $H$ is \textit{circulant} if and only if $h_{i,j} = h_{i+1, j+1}$ for all $0 \leq i,j \leq n-1$ with indices interpreted modulo $n$. It is well-known that $H$ is circulant if and only if all columns of the $n \times n$ Fourier matrix $F(C_{n})$ are eigenvectors of $H$, \cite{Davis}. Conversely, any bent function for $F(C_{n})$ corresponds to a circulant Hadamard matrix, as in the following theorem.

\begin{theorem}
Suppose that ${\bf x}$ is $F(C_{n})$-bent. Then $H = [ \textbf{x}_{i-j} ]_{0 \leq i,j\leq n-1}$ is a circulant Hadamard matrix, where indices are interpreted modulo $n$.
\end{theorem}

\begin{proof}
The entries of $H$ have unit norm, and the eigenvalues of $H$ are the entries of $F(C_{n}){\bf x}$, which have norm $n$. Such a matrix is necessarily Hadamard, \cite[Section 1]{Maxdet}.
\end{proof}

As an illustration of the utility of bent vectors for studying circulant Hadamard matrices, we give a new proof of the following previously-known result.

\begin{corollary}
There does not exist a real circulant Hadamard matrix of order $n = 4p^{2}$ for any prime $p \equiv 3 \mod 8$.
\end{corollary}

\begin{proof}
Consider $H \in \BH(n,n)$, where $H$ is circulant. Then any column ${\bf x}$ of $F(C_{n})$ is an eigenvector of $H$, i.e., is $H$-bent, and the corresponding eigenvalue is $\lambda = (H{\bf x})_{1}$.

We observe that $4p^{2}$ is self-conjugate modulo $4p^{2}$: by hypothesis $p$ is self-conjugate modulo $4$ since $p \equiv 3 \mod 4$. Using standard results on quadratic reciprocity, $2$ is a quadratic non-residue modulo $p$ when $p \equiv 3 \mod 8$. Hence $2^{\frac{p-1}{2}} = ap-1$ for some integer $a$. It follows that
\[ \left(2^{\frac{p-1}{2}}\right)^{p} \equiv (ap-1)^p \equiv -1 \mod p^{2}. \]
Hence there exists an integer $f$ such that $2^{f} \equiv -1 \mod p^{2}$. Theorem \ref{teorema_principal} now implies that $\lambda/{2p}$ is a root of unity. If the eigenvalues of $\frac{1}{2p}H$ are all roots of unity, then this matrix has finite multiplicative order. This contradicts a result of Craigen and Kharaghani showing that a real circulant Hadamard matrix of order $n > 4$ (scaled to be unitary) cannot have finite multiplicative order \cite[Lemma 4]{CraKha93}.
\end{proof}

Hiranandani and Schlenker classified the circulant matrices in $\BH(n,k)$ up to equivalence under cyclic permutations and scalar multiplication (by a $k^{\rm th}$ root of unity) for all $n,k \leq 12$ (and some larger values), \cite{HirSch}. Note that this is stronger than Hadamard equivalence as described in Section \ref{sec:Butson}. They found 519 inequivalent circulant Butson matrices. For each matrix $M \in \BH(n,k)$ in this classification, we found that the multiplicative order of $\frac{1}{\sqrt{n}}M$ divides $\mathrm{lcm}(4,n,k)$. Hence every eigenvalue of $\frac{1}{\sqrt{n}}M$ is a root of unity. Denoting by ${\bf x}$ the transpose of the first row of $H$ and by $\sqrt{n}{\bf y}$ the corresponding vector of eigenvectors, we find many examples of bent pairs for the Fourier matrices $F(C_{n})$. Curiously, there is currently no known example of a circulant Hadamard matrix $H \in \BH(n, k)$ such that $\frac{1}{\sqrt{n}}H$ has infinite multiplicative order. This motivates the following question, posed for more general abelian groups.

\begin{question}\label{q:Root}
Let $G$ be an abelian group of order $n$, and suppose that ${\bf x}$ is $F(G)$-bent. Must the entries of $\frac{1}{\sqrt{n}}F(G) {\bf x}$ be roots of unity?
\end{question}

A positive answer to Question \ref{q:Root}, together with the result of Craigen and Kharaghani quoted above would imply a proof of Ryser's circulant Hadamard conjecture: that a real circulant Hadamard matrix must have order $4$.

\section{Construction of conjugate self-dual bent vectors}\label{constructions}
 In this section, we will focus on the equation
\begin{equation}
H\mathbf{x}=\sqrt{n} \overline{\mathbf{x}}.
\end{equation}
Concretely, we will explicitly provide conjugate self-dual $H$-bent vectors $X$ for some matrices $H\in \BH(n,k)$. 
Representatives of three equivalence classes of $\BH(9,3)$
are provided in \cite{EFO15}. These may be accessed online at \url{https://www.daneflannery.com/classifying-cocyclic-butson-hadamard-matrices}. We verified computationally that there exist conjugate self-dual $H$-bent vectors for each of these representatives.

\subsection{Bent vectors for tensor products}

Define $\Phi$ to be the vectorisation map, i.e. $\Phi(H) = {\bf x}$ with $x_{(i-1)n+j}=h_{i,j}$. First we give an easy composition result.

\begin{lemma}
    If $\mathbf{x}$ and $\mathbf{y}$ are conjugate self-dual $H$-bent and $K$-bent respectively, for $H\in \BH(n,k)$ and $K\in \BH(m,k)$, then $\mathbf{x}\otimes \mathbf{y}$ is conjugate self-dual $(H\otimes K)$-bent.
\end{lemma}

\begin{proposition}\label{prop-sd-fbh}
Let $A, B, M$ be $n\times n$ matrices. Then $ (A\otimes B)^*\Phi(M)=\Phi(A^*\; M \;\overline{B})$.
\end{proposition}

\begin{proof}
Let $\mathbf{x} = \Phi(M)$, so $x_{(i-1)n+j} = m_{i,j}$. Recall that the Kronecker product is bilinear and commutes with complex conjugation, hence $(A\otimes B)^{*}=A^* \otimes B^*$. By direct computation, the following products are equal:
\[\begin{array}{cc}
  \Big[ (A\otimes B)^*\Phi(M)\Big]_{(j-1)n+l}=\displaystyle\sum_i \bar{a}_{i,j}\sum_k \bar{b}_{k,l}x_{(i-1)n+k}, &  1\leq j,l\leq n;\\[5mm]
  \Big[ A^*\; M \;\overline{B}\Big]_{j,l}=\Big[A^*\; (M \;\overline{B})\Big]_{j,l}=  \displaystyle\sum_i \bar{a}_{i,j}\sum_k \bar{b}_{k,l}x_{(i-1)n+k}, &  1\leq j,l\leq n.
\end{array}
\]
The claim follows. 
\end{proof}

This result may be specialised in various ways; essentially it shows that tensor products often admit bent vectors. 

\begin{corollary}\label{tensors}
\begin{enumerate} 
\item If $H$ is Hadamard then the vector $\Phi(H)$ is conjugate self-dual $(H^{\ast}\otimes H^{\ast})$-bent.
\item If $H$ and $M$ are commuting Hadamard matrices then $\Phi(M)$ is self-dual $(H\otimes \overline{H})$-bent. 
\item If $H$ and $M$ are amicable Hadamard matrices, that is $HM^{\ast} = MH^{\ast}$, then $\Phi(M)$ is conjugate self-dual $(H\otimes H^{\top})$-bent if $M$ is symmetric.
\end{enumerate} 
\end{corollary}

\begin{proof}
\begin{enumerate} 
\item Set $A = B = H^{\ast}$ and $M = H$ in Proposition \ref{prop-sd-fbh}. Then:
\[ \frac{1}{n}\Big(H\otimes H\Big)^{\ast} \Phi(H)=\frac{1}{n}\Phi(H^* H \overline{H})=\Phi(\overline{H})\,.\]
\item Since $H$ and $M$ are normal matrices, $H^{\ast}$ commutes with $M$, and thus 
\begin{align*}
M &= H^{\ast}M(H^{\ast})^{-1}\\
&= \frac{1}{n}H^{\ast}MH.
\end{align*}
Set $A = H, B = \overline{H}$ in Proposition \ref{prop-sd-fbh}, then $(H \otimes \overline{H})\Phi(M) = \Phi(M)$.
\item By hypothesis 
\begin{align*}
M^{\ast} &= H^{-1}MH^{\ast}\\
&= \frac{1}{n}H^{\ast}MH^{\ast}.
\end{align*}
Set $A = H, B = H^{\top}$ in Proposition \ref{prop-sd-fbh}, then $(H \otimes H^{\top})\Phi(M) = \Phi(M^{\ast})$. But $M$ is symmetric so $M^{\ast} = \overline{M}$ as required. \qedhere
\end{enumerate} 
\end{proof}

Since there exists $H \in \mathrm{BH}(6, 3)$, by Corollary \ref{tensors}, there exists matrices in $\mathrm{BH}(36, 3)$ with self-dual and conjugate self-dual vectors. 

\subsection{Many bent vectors from Bush-type matrices}

\emph{Bush-type} Hadamard matrices were introduced by K. A. Bush in 1971 to study finite projective planes, but have since found many other applications, \cite{bush-71}. Recently complex analogues of Bush-type Hadamard matrices have been studied, \cite{htcv}. 

\begin{definition} [Bush, \cite{bush-71}]
A matrix $H \in \BH(n^2,k)$ is of \emph{Bush-type} if it may be subdivided into $n \times n$ blocks $H_{ij}$ such that $JH_{ij} = H_{ij}J = \delta_{i,j}nJ$, for all $1 \leq i,j \leq n$, where $J$ is the $n \times n$ matrix of all ones.
\end{definition}

For $k=2$, this reduces to the definition of a (real) Bush-type Hadamard matrix. A key property of Bush-type matrices is that the diagonal block matrices $H_{ii}$ can be multiplied by arbitrary $u \in \langle \zeta_{k}\rangle$ while preserving the orthogonality of rows and columns. 

\begin{proposition}
    If $H\in \BH(4m^2,4)$ is of Bush-type, then there are at least $2^{2m}$ self-dual $H$-bent vectors, and at least $2^{2m}$ conjugate self-dual $(-H)$-bent vectors.
\end{proposition}
\begin{proof}
Let ${\bf 1}$ be the all-one vector of length $2m$, and let $u_k\in\{\pm \zeta_{4}\}$ be arbitrary. Any vector ${\bf x}^{\top}=(u_1 {\bf 1},\ldots, u_{2m}{\bf 1}),$ is both self-dual $H$-bent and conjugate self-dual $(-H)$-bent. This is a straightforward consequence of $H$ being of Bush type.
\end{proof}

A corresponding result holds when $k$ is odd. 

\begin{proposition}\label{bush-oddk}
Let $k$ be odd and let $H$ be a Bush-type $\BH(n^2,k)$, then exist $k^{n}$ matrices in $\BH(n^2,k)$, each admitting a self-dual and a conjugate self-dual bent vector.
\end{proposition}
\begin{proof}
Let $H=[H_{ij}]$ and 
\[ H'=[H'_{ij}]=\left\{\begin{array}{cr} H_{ij} & i\ne j\\ \zeta_{k}^{u_{i}}H_{ij} & i=j, \end{array}\right.\]
where $u_i\in \Z_{k}$ for each $i$. Then $H'$ is a $\BH(n^2,k)$ and the vector ${\bf x}$ defined by  ${\bf x}^{\top}=(\zeta_{k}^{u_{1}\alpha} {\bf 1},\ldots,\zeta_{k}^{u_{n}\alpha} {\bf 1})$  is self-dual $H'$-bent when $\alpha = \frac{k+1}{2}$, and is conjugate self-dual $H'$-bent when $\alpha = \frac{k-1}{2}$.
\end{proof}

For an odd prime power $q$, there exists $H \in \BH(2q,q)$ as a Corollary to Theorem 2.4 of Jungnickel, \cite{jung}.  It follows from Theorem 5 of \cite{htcv} that there is a Bush-type matrix in $\BH(4q^2,q)$ admitting many self-dual bent vectors.

Given that many distinct self-dual bent vectors for a Bush-type Hadamard matrix, one may ask whether there exists an orthonormal basis of bent vectors; in fact such objects have been previously studied as \emph{unbiased Butson Hadamard matrices}.

\begin{definition} Two $\BH(n,k)$ matrices $H$ and $K$ are called \emph{unbiased} if $HK^*=zL$ for some $\BH(n,k)$ matrix $L$ where $zz^{\ast} = n$.
\end{definition}

It follows from the definition that for a pair of unbiased $\BH(n^2,k)$ matrices $H$ and $K$, each column of $K$ is a bent vector for $H$, and vice versa.

\subsection{Conjugate self-bent Butson matrices}

We conclude this section on explicit constructions of conjugate self-dual bent vectors by studying the matrix equation $MN = p\overline{N}$ in dimension $p^2$ for an odd prime $p$. We demonstrate a construction of infinitely many pairs of matrices $M$ and $N$ with this property, achieved by building Bush-type Butson matrices in $\textrm{BH}(p^{2}, p)$ from the Fourier matrices $F(C_{p})$ for any odd prime $p$. Moreover, we consider the special case where $M^{2} = p\overline{M}$, i.e., matrices $M$ with columns comprised entirely of conjugate self-dual $M$-bent vectors. We term such matrices \textit{conjugate self-bent} and produce examples in $\BH(9,3)$. 

%\begin{example}\label{third-root}
%For $p=3$ and $\zeta = \zeta_{3}$, let
%\[
%C_0=\begin{bmatrix} 1 & 1 & 1\\1 & 1 &1\\1 & 1 &1 \end{bmatrix},~ C_1=\begin{bmatrix} 1 & \zeta & \zeta^2\\\zeta^2  & 1 &\zeta\\\zeta & \zeta^2 &1 \end{bmatrix},~ C_2=\begin{bmatrix} 1 & \zeta^2 & \zeta\\\zeta  & 1 &\zeta^2\\\zeta^2 & \zeta &1 \end{bmatrix},~ \text{and}~A=\begin{bmatrix} C_0 & C_1 & C_2\\C_2 & C_0 & C_1\\ C_1 & C_2 &C_0 \end{bmatrix}. 
%\]
%Then $A^{2}=3\overline{A}$.
%\end{example}

 We begin with a lemma on a set of orthogonal projection matrices constructed from the rows of $F(C_{p})$. 

\begin{lemma}\label{b2}
Denote by $r_{a}$ the $a^{\textrm{th}}$ row of $F(C_{p})$ for $0 \leq a \leq p-1$, and let $R_{a} = r_{a}^{\ast} r_{a}$ be the outer product. 
The following hold: 
\begin{enumerate}
\item Rank 1 projectors: $R_{a}$ is rank $1$, and satisfies $R_{a}^{2} = pR_{a}$ for $0 \leq a \leq p-1$. 
\item Hermitian: $R_a^*=R_a$, for each $a$. Furthermore, $R_{a}^{\top} = R_{p-a}$. 
\item Orthogonal: $R_aR_b=0$, $a\neq b$. 
\item Basis: $\sum_{0}^{p-1}R_a^2=p^{2}I_p$. 
\end{enumerate}
\end{lemma}

Next we construct a set of matrices in $\textrm{BH}(p^{2}, p)$ with the $R_{a}$ as blocks. 

\begin{proposition} \label{prop-conj-dual-vectors-bush}
Let $B_{a}$ be the block-circulant matrix having $[ R_{0}, R_{a}, R_{2a}, \ldots, R_{(p-1)a}]$ as its first row, with indices interpreted modulo $p$. For $a=1,2,\ldots,p-1$, the matrix $B_a$ is a symmetric Bush-type Butson Hadamard matrix $\BH(p^2,p)$. 
Furthermore, $\overline{B_a}=B_{p-a}$ and $B_aB_{(p-2)a}=pB_{2a}=p\overline{B}_{(p-2)a}$.
%\marginpar{\tiny \textcolor{blue}{$\overline{B}_{p-2a}$ should be $\overline{B}_{(p-2)a}$. Right??? RE, correct, amended}}
\end{proposition} 
%\marginpar{\tiny{\textcolor{blue}{Proposition \ref{prop-conj-dual-vectors-bush} provides examples of $K^2=p\overline{K}$ ONLY if $p=3$. Right???? RE: Ok, I see now. Yes this is examples of matrices $M$ and $N$ such that $MN = \overline{N}$. This still produces conjugate self-dual bent vectors for $M$, the columns of $N$, but the $M^{2} = \overline{M}$ property only when $p=3$.}}}
\begin{proof} 
Consider the inner product of two rows $i$ and $j$ of blocks from $B_{a}$: 
\[ \sum_{k=0}^{p-1} R_{(p-i+k)a}R_{(p-j-k)a}^{\ast} = \delta_{i,j} p^{2} I_{p}\,,\]
which follows from Lemma \ref{b2} since every product consists of distinct blocks (and so is zero) if $i \neq j$; otherwise the sum evaluates to a scalar multiple of the identity matrix. Hence $B_{a}$ is a Bush-type Butson Hadamard matrix. 

The matrix is symmetric (but not Hermitian) since the $(i,j)$ entry of $B_{a}$ is $R_{(p-i+j)a}$ while the $(i,j)$ entry of the transpose is $R_{(p-j+i)a}^{\top} = R_{(p-i+j)a}$ by Lemma \ref{b2}. The proof that $\overline{B_{a}} = B_{p-a}$ is similar. With reference to Lemma \ref{b2}, the product of the $i^{\rm th}$ row of blocks from $B_{a}$ with the $j^{\rm th}$ column of blocks from $B_{(p-2)a}$ is
\[
\sum_{k=0}^{p-1} R_{(k-i)a}R_{(2k-2j)a} = R_{(2j-2i)a}^{2} = pR_{(2j-2i)a}. 
\]
It follows that $B_{a}B_{(p-2)a} = pB_{2a}$, as required.
\end{proof}

As a consequence of Proposition \ref{prop-conj-dual-vectors-bush}, for any prime $p \geq 3$ and any choice of $a \in \{1,\ldots,p-1\}$ the columns of $B_{(p-2)a}$ are conjugate self-dual $B_{a}$-bent vectors. Further, in the special case where $p=3$, it follows that $B_{a}$ is conjugate self-bent.

\begin{example}\label{third-root}
For $p=3$ and $\zeta = \zeta_{3}$, let
\[
R_0=\begin{bmatrix} 1 & 1 & 1\\1 & 1 &1\\1 & 1 &1 \end{bmatrix},~ R_1=\begin{bmatrix} 1 & \zeta & \zeta^2\\\zeta^2  & 1 &\zeta\\\zeta & \zeta^2 &1 \end{bmatrix},~ R_2=\begin{bmatrix} 1 & \zeta^2 & \zeta\\\zeta  & 1 &\zeta^2\\\zeta^2 & \zeta &1 \end{bmatrix},~ \text{and}~B_1=\begin{bmatrix} R_0 & R_1 & R_2\\R_2 & R_0 & R_1\\ R_1 & R_2 & R_0 \end{bmatrix}. 
\]
Taking $M=B_1$, then $M^{2}=3\overline{M}$.
\end{example}

\section{On the covering radius of Butson codes}\label{codes}

A code over $\Z_k$ (or $\Z_k$-code) of length $n$ is a nonempty subset $C$ of $\Z_k^n$. The elements of $C$ are called \textit{codewords}. The Hamming weight of a vector ${\bf v}\in \Z_k^n$, denoted by $\wt({\bf v})$, is the number of nonzero coordinates of ${\bf v}$. The Hamming distance between two vectors ${\bf v}, {\bf w}\in \Z_k^n$, denoted by $d({\bf v},{\bf w})=\wt({\bf v-w})$, is the number of coordinates in which they differ. The minimum Hamming distance of a code $C$ is given by $d = \min_{{\bf x},{\bf y} \in C, {\bf x}\neq{\bf y}} d({\bf x},{\bf y})$. When this parameter is known we say $C$ is a $(n,|C|,d)$ code. The {\em covering radius} of a $\Z_k$-code $C$ of length $n$ is defined by the formula
\[r(C)=\max_{{\bf x}\in \Z_k^n} \min_{{\bf y}\in C} d({\bf x},{\bf y}).\]

Let $V$ be a free module of dimension $m$ over $\mathbb{Z}_{q}$ and denote by $\mathrm{Fun}(V)$ the space of functions $f: V \rightarrow \mathbb{Z}_{q}$, which is of dimension $q^{m}$ over $\mathbb{Z}_{q}$. The first order generalised Reed-Muller code $R_{q}(1, m)$ is a submodule of $\mathrm{Fun}(V)$ generated by the linear functions $\mathbb{Z}_{q}^{m} \rightarrow \mathbb{Z}_{q}$, which is easily seen to be a $(q^{m}, q^{m+1}, q^{m-1})$-code. 

The {\em nonlinearity} of a function $f\colon {\mathbb{Z}}_{q}^m \rightarrow {\mathbb{Z}}_{q}$ is the Hamming distance between $f$ and $R_{1}(q, m)$. Thus the maximal non-linearity of a function $f:\mathbb{Z}_{q}^{m} \rightarrow \mathbb{Z}_{q}$ is equal to the covering radius of $R_{q}(1, m)$, this quantity is denoted $\rho_{q}(m)$. Schmidt showed that when $m$ is even and $q$ is prime, then $\rho_q(m)=q^{m-1}(q-1)-q^{m/2-1}$, \cite{Sch20}.

\begin{definition}
Let $H\in \BH(n,k)$. We denote by $R_H$ the $\Z_k$-code of length $n$ consisting of the rows of $L(H)$, and we denote by $C_H$ the  $\Z_k$-code defined as
$C_H=\cup_{\alpha \in \Z_{k}}(R_H+\alpha {\bf 1})$. The code $C_H$ over $\Z_k$  is called a \textit{Butson Hadamard code} (briefly, BH-code).
\end{definition} 

\begin{example} Let $H\in \BH(4,8)$ be as in Example \ref{formaLog}. Then $R_{H}$ consists of the following four vectors: $\{[0,  0, 0,  0],[0, 2,  4, 6],[0, 4, 0, 4], [0, 6, 4, 2]\}$, and $C_{H}$ is the following set of $32$ codewords: 
\[ C_{H}=\left\{ [i,  i, i,  i], [i, 2+i,  4+i, 6+i], [i, 4+i, i, 4+i], [i, 6+i, 4+i, 2+i] \mid i \in \mathbb{Z}_{8} \right\}\] 
Since the rows of $H$ are orthogonal, no two distinct rows can share more than two entries; it follows that this is a $(4, 32, 2)$-code.
\end{example}

We study upper and lower bounds on the covering radius of $C_{H}$ for an arbitrary $H \in \BH(n, q)$, and under the hypothesis that $H$ admits a bent vector. An upper bound was obtained by Leducq from a relatively weak combinatorial condition.

\begin{theorem}[cf \cite{Leducq}]\label{Leducq6}
Let $H\in \BH(n, q)$ where $q$ is an odd prime. 
Then 
\[ r(C_{H}) \leq \frac{q-1}{q}n - \frac{1}{q}\sqrt{n}\,.\]
\end{theorem} 

\begin{proof} 
This follows from Leducq's Theorem 6, \cite{Leducq} once the conditions of that theorem are verified. We observe that a code is \text{self-complementary} if it is closed under addition of $\alpha \bf{1}$ for all $\alpha \in \mathbb{Z}_{q}$ and this condition is clearly met by the construction of $C_{H}$.

A code is of strength $2$ if in any two coordinates of the code, one sees each pair $[x,y] \in \mathbb{Z}_{q}^{2}$ equally often when ranging over all codewords. Since $q$ is prime, the only vanishing sum of $q^{\textrm{th}}$ roots of unity is the full sum, and the result follows from orthogonality of columns of the Hadamard matrix $H$. 
\end{proof} 

It is well known in the literature that quadratic functions are often maximally non-linear: when $q = 2$ the classes of bent functions and maximally non-linear functions co-incide; as previously observed the construction of Kumar-Scholtz-Welch is via quadratic functions. In fact, Leducq uses precisely this construction to obtain the following result. 

\begin{theorem}[Theorem 2, \cite{Leducq}]
Let $H = F(C_{q}^{m})$. Then 
\[ \frac{q-1}{q} q^{m} - \frac{1}{q} q^{\lceil m/2 \rceil} \leq \rho(C_{H}) \leq \frac{q-1}{q} q^{m} - \frac{1}{q} q^{ m/2}.\]
\end{theorem} 

In particular, when $m$ is even, this bound is tight. But when $m$ is odd, bent functions are typically not the functions of maximal non-linearity. We consider a generalisation of this result to more general Buston matrices; first we restrict to $k = 3$. 

\begin{lemma}\label{lemma-dh}
 Let ${\bf v},{\bf w}\in \langle \zeta_{3}\rangle^n$. Then $d( L({\bf v}),L({\bf w}))=2/3 \Big(n-\mathcal{R}(\langle {\bf v},{\bf w}\rangle)\Big)$.
\end{lemma}%\marginpar{\tiny \color{blue} The reviewer's comment here was correct so I used the $L$ notation to go between complex numbers and $\mathbb{Z}_{3}$.}

\begin{proof}
Observe that
\[ \langle {\bf v},{\bf w}\rangle=s_0 +s_1\zeta_3+s_2\zeta_3^2=s_0-\frac{s_1+s_2}{2}+ i\, \frac{\sqrt{3} \,(s_1-s_2)}{2}\]
where $n=s_0+s_1+s_2$ and $d(L({\bf v}), L({\bf w}))=s_1+s_2$.
Then $n-\mathcal{R}(\langle {\bf v},{\bf w}\rangle)=3/2 \, \,d(L({\bf v}),L({\bf w}))$, and the result follows.
\end{proof}

\begin{theorem} 
Suppose that $H \in BH(n, 3)$ admits a bent vector. Then 
\[ \frac{2}{3}n - \frac{2}{3}\sqrt{n}\leq r(C_{H}) \leq \frac{2}{3}n - \frac{1}{3}\sqrt{n}.\]
\end{theorem} 

\begin{proof} 
The upper bound comes from Theorem \ref{Leducq6} with $q = 3$. 
For the lower bound, let ${\bf x}$ be a bent vector for $H$. By hypothesis $|\langle {\bf x}, r_{a}\rangle|^{2} = n$ for any $r_{a}$, hence $\mathcal{R}(\langle {\bf x}, r_{a} \rangle) \leq \sqrt{n}$. By Lemma \ref{lemma-dh}, we conclude that 
\[ d({\bf x}, r_{a}) \geq \frac{2}{3}( n - \sqrt{n}) \]
for any codeword $r_{a}$. Thus ${\bf x}$ is a vector at distance at least $\frac{2}{3}( n - \sqrt{n})$ from all codewords, which concludes the proof. 
\end{proof}

\section*{Acknowledgement}
The first author was supported by
 project TED2021-130566B-I00  from the Ministry of Science and Innovation of the Government of Spain.

The fourth author acknowledges the support of the Faculty of Humanities and Social Sciences at DCU and the support of the CPID in TUS.

\end{document}